\documentclass[11pt]{article}
\usepackage{amsmath, amssymb}
\usepackage{mathtools}
\usepackage{algorithm}
\usepackage{algpseudocode}
\usepackage{fmtcount}
\usepackage{graphicx}
\usepackage{enumitem} 
\usepackage{subcaption}
\usepackage{afterpage}
\usepackage[noadjust]{cite}
\usepackage[font=small,labelfont=bf]{caption}
\usepackage[labelfont=bf]{caption}
\usepackage{setspace}
\allowdisplaybreaks
\usepackage[title,titletoc]{appendix}
\usepackage{epsfig}
\usepackage{graphicx}
\usepackage{hhline}
\usepackage{latexsym}
\usepackage{amssymb}
\usepackage{amsmath,amscd}
\usepackage{multirow} 
\usepackage{array}
\usepackage{caption}
\usepackage{subcaption}
\usepackage{authblk}
\usepackage{color}
\usepackage{natbib}
\usepackage{rotating}
\usepackage{graphicx,lscape}
\usepackage{amsmath, amsthm, amssymb}
\usepackage{graphicx}
\usepackage{epsfig}
\usepackage{graphics,color, graphicx}
\usepackage{latexsym}
\usepackage{fullpage}
\usepackage{booktabs,fixltx2e}
\usepackage[flushleft]{threeparttable}
\usepackage{amssymb,amsmath,amscd}
\usepackage{fancyvrb}
\usepackage{pdfpages}
\usepackage{fmtcount}
\usepackage{url}
\usepackage[margin=1in]{geometry}

\theoremstyle{plain}
\newtheorem{proposition}{Proposition}[section]
\newtheorem{theorem}[proposition]{Theorem}
\newtheorem{corollary}[proposition]{Corollary}
\newtheorem{lemma}[proposition]{Lemma}
\theoremstyle{definition}

\title{\bf Radial fields on the manifolds of symmetric positive definite matrices
	\medskip
}

\author[1,2]{Ha-Young Shin}
\affil[1]{
	Department of Statistics and Actuarial Science, Soongsil University
}
\affil[2]{
	Integrative Institute of Basic Sciences, Soongsil University
}
{
    \makeatletter
    \renewcommand\AB@affilsepx{: \protect\Affilfont}
    \makeatother

    \affil[ ]{Email}

    \makeatletter
    \renewcommand\AB@affilsepx{, \protect\Affilfont}
    \makeatother

    \affil[1,2]{hayoung.shin@gmail.com}
}
\date{}

\begin{document}
	\maketitle
	
	\begin{abstract}
On Hadamard manifolds, the radial fields, which are the negative gradients of the Busemann functions, can be used to designate a canonical sense of direction. This could have many potential applications to Hadamard manifold-valued data, for example in defining notions of quantiles or treatment effects. Some of the most commonly encountered Hadamard manifolds in statistics are the spaces of symmetric positive definite matrices, which are used in, for example, covariance matrix analysis and diffusion tensor imaging. Surprisingly, an expression for the radial fields on these manifolds is unavailable in the literature even though the issue arises quite naturally when studying the geometry of these spaces. This paper aims to fill this gap by deriving such an expression, and also demonstrates their smoothness.
\end{abstract}

\section{Introduction}

In a metric space $(M,d)$, two unit-speed geodesic rays $\gamma_1,\gamma_2:[0,\infty)\rightarrow M$ are called asymptotic if $d(\gamma_1(t),\gamma_2(t)), t\in[0,\infty)$, is bounded; in the rest of this paper, we will refer to the metric space $(M,d)$ as simply $M$, with its metric $d$ being implicit. One can form an equivalence relation among unit-speed geodesic rays in $M$ on the basis of their being asymptotic; the set of all resulting equivalence classes is called the boundary at infinity $\partial M$, not to be confused with the topological boundary. There is a class of metric spaces called Hadamard spaces, equivalently complete CAT(0) spaces or global non-positive curvature spaces, with the following property: in a Hadamard space $M$, for all $\xi\in \partial M$ and $x\in M$, there is a unique $\gamma\in\xi$ satisfying $\gamma(0)=x$; see Chapter II.8 of \cite{Bridson1999} for more information.

Hadamard spaces that are also Riemannian manifolds are called Hadamard manifolds, which can equivalently be characterized as complete, simply connected Riemannian manifolds whose sectional curvatures are non-positive. By the Cartan--Hadamard theorem, an $n$-dimensional Hadamard manifold $M$ is diffeomorphic to $\mathbb{R}^n$ via the exponential map $\exp_p:T_pM\cong\mathbb{R}^n\rightarrow M$ at any $p\in M$. For any $x\in M$ and $\xi\in \partial M$, denoting the unique member of $\xi$ originating at $x$ by $\gamma_x$, we can associate a unique unit vector $\xi_x:=\gamma_x'(0)$ in $T_xM$ with $\xi$. The vector fields on $M$ defined by $x\mapsto \xi_x$ for $\xi\in \partial M$, which have been called radial fields (\cite{Heintze1977}, \cite{Shcherbakov1983}), are also the negative gradients of the so-called Busemann functions $x\mapsto \lim_{t\rightarrow\infty} d(x,\gamma(t))-t$, where $\gamma$ is any member of $\xi$; thus the radial fields are normal to the level sets of these function, which are called horospheres or horocycles. Proposition 3.3(a) of \cite{Shin2023} shows that 
\begin{equation}\begin{aligned} \label{xi}
    \xi_x=\lim_{t\rightarrow\infty}\frac{\log_x(\gamma(t))}{d(x,\gamma(t))},
\end{aligned}\end{equation}
where $\gamma:[0,\infty)\rightarrow M$ is any member of the equivalence class $\xi$, and Proposition 5.1 in that paper presents an expression for the radial fields in hyperbolic spaces. A note about the notation in this paper: $\exp$ and $\log$ with subscripts denote the Riemannian exponential maps and their inverses, respectively, while $\exp$ and $\log$ without subscripts denote the usual exponential and logarithm for real and positive numbers and $\mathrm{Exp}$ and $\mathrm{Log}$ denote the matrix exponential and logarithm.

Radial fields can be used to define a canonical sense of direction on Hadamard manifolds. That is, we can talk about $\xi_x$ being the unit vector at $x$ ``in the direction of $\xi$''. This is canonical in the sense that it does not require arbitrary decisions. On the other hand, one might try to define direction using parallel transport, but this is problematic because parallel transport between two points depends on the path taken between those points. One way to deal with this might be to choose a base point and corresponding tangent space to which we transport vectors from other points, but in general, this choice of base point would be arbitrary.

Radial fields are interesting mathematical objects in their own right and as tools for studying the boundary at infinity, Busemann functions, and horospheres on Hadamard manifolds; beyond this, they also have many potential applications due to providing this sense of direction. As an example, consider the problem of defining quantiles for Hadamard space-valued data. In the univariate case, like means and medians, quantiles can be expressed as minimizers of loss functions; the $\alpha$-quantile, where $\alpha\in(0,1)$, of a real random variable $X$ can be defined as $\arg\min_{p\in\mathbb{R}}|X-p|\{(1-\alpha)I(X\leq p)+\alpha I(X>p)\}$. By making the transformation $u=2\alpha-1$, $2[|X-p|\{(1-\alpha)I(X\leq p)+\alpha I(X>p)\}]=|X-p|+u(X-p)$, and quantiles can be indexed by $u\in (-1,1)$, the open 1-dimensional ball of radius 1; thus \cite{Chaudhuri1996} defined the multivariate $u$-quantile, where $u$ is a fixed vector of norm less than 1, of a random $n$-vector $X$ to be $\arg\min_{p\in\mathbb{R}^n}\lVert X-p\rVert+\langle u,X-p\rangle$. By conceptualizing $u=\lVert u\rVert (u/\lVert u\rVert)$ (if $u\neq 0$) and $X-p$ as tangent vectors at $p$, \cite{Shin2023} generalized quantiles to Hadamard manifold-valued data by indexing these quantiles with some $(\beta,\xi)\in[0,1)\times \partial M$ and defining the $(\beta,\xi)$-quantile of an $M$-valued random element $X$ to be $\arg\min_{p\in M}d(p,X)+\langle \beta\xi_p,\log_p(X)\rangle$. As noted in \cite{Shin2025}, observing that $\lVert X-p\rVert+\langle u,X-p\rangle=\lVert p-X\rVert-\langle u,p-X\rangle$, we could alternatively generalize quantiles to Hadamard manifolds with $\arg\min_{p\in M}d(X,p)-\langle \beta\xi_x,\log_X(p)\rangle$. Other asymmetric loss functions, such as the expectile (\cite{Newey1987}, \cite{Hermann2018}) or M-quantile (\cite{Breckling1988}, \cite{Konen2022}) loss functions can analogously be generalized to Hadamard manifolds using radial fields.

Another example is in the area of causal inference. The most important parameter in causal inference is the average treatment effect (\textsc{ATE}) $E(r_T)-E(r_C)$, where $r_T$ and $r_C$ are random variables representing the potential outcomes for the same subject with and without treatment, respectively. Consider the space $[0,\infty)\times \partial M/\sim$, where the equivalence relation is defined by $(\beta,\xi)\sim(\beta',\xi')$ if $\beta=\beta'=0$ or $(\beta,\xi)=(\beta',\xi')$; this space is called the (Euclidean) cone over the boundary $\partial M$ (\cite{Bertrand2012}, \cite{Hirai2024}). Then on Hadamard manifolds, \cite{Dissertation} defines an \textsc{ATE} to be a $[(\beta,\xi)]\in [0,\infty)\times \partial M/\sim$ for which $\exp_{r_C}(\beta\xi_{r_C})$ and $r_T$ have the same Fr\'echet means; the quantile and median treatment effects could be defined analogously. \cite{Dissertation} also defines an individual treatment effect as that $[(\beta,\xi)]\in [0,\infty)\times \partial M/\sim$ for which $\exp_{r_C}(\beta\xi_{r_C})=r_T$ and subsequently defines the homogenous treatment effect model, a crucial model in traditional causal inference, as one in which the individual treatment effect is constant.

Many other statistical tools that use vectors could also be generalized to Hadamard manifold-valued data by using non-negative numbers and radial fields to define magnitudes and directions, respectively, but because the use of radial fields for statistical inference on Hadamard manifolds is a new area of research, much of this vast potential is yet unexplored.

An expression for the radial fields on the spaces of symmetric positive definite matrices specifically is needed because these are some of the most commonly encountered examples of Hadamard manifolds. Because the issue of radial fields naturally arise when studying the geometry of Hadamard manifolds, we were surprised to be unable to find such an expression in the literature. This paper aims to fill this gap.

Denote the space of real symmetric $m\times m$ matrices by $\mathcal{S}_m$ and the space of real symmetric positive-definite (SPD) $m\times m$ matrices by $\mathcal{P}_m$. The former is an $m(m+1)/2$-dimensional vector space, and the latter can be considered a $m(m+1)/2$-dimensional smooth manifold on which the tangent space at each point is isomorphic to $\mathcal{S}_m$. This manifold is typically equipped with one of a handful of different Riemannian metrics, such as the Log-Cholesky metric of \cite{Lin2019}, but the most commonly used is the so-called trace, or affine invariant, metric, defined at $x\in \mathcal{P}_m$ by
 \begin{align*}
     \langle v_1,v_2\rangle=\text{tr}(x^{-1}v_1x^{-1}v_2),
 \end{align*}
\sloppy where $v_1,v_2\in T_x\mathcal{P}_m\cong\mathcal{S}_m$. This Riemannian manifold is complete and simply connected with sectional curvatures in $[-1/2,0]$ (see Proposition I.1 of \cite{Criscitiello2020}); therefore, this is a Hadamard manifold.

These spaces have many uses, and often, data take values in them. For example, diffusion tensor imaging (DTI), first proposed by \cite{Basser1994}, is a methodology for modeling diffusion of water molecules in voxels of brain scans as $3\times 3$ SPD matrices lying in $\mathcal{P}_3$. Crucially, these spaces need to be studied because covariance matrices (and their inverses, precision matrices), among the central objects of study in statistics and probability, are SPD matrices. Covariance matrices can be random $\mathcal{P}_m$-valued objects in their own right, for example, as sample covariance matrices or as parameters in a Bayesian framework, in which case the assigned prior is most commonly the inverse-Wishart distribution (see, for instance, \cite{Lee2018}).

Our main contribution here is an expression for the radial fields on $\mathcal{P}_m$, which is much less forthcoming than in the case of hyperbolic space. We also demonstrate that the radial fields are smooth on $\mathcal{P}_m$, which is not known to be true in general on Hadamard manifolds.

For an example of an application of the results in this paper, see Figure \ref{fig}, which comes from \cite{Shin2025}. Because an SPD matrix is diagonalizable, its eigenvectors can be used to define the axes of an ellipsoid and its eigenvalues their lengths; thus DTI data can be visualized as ellipsoids. Pictured on the left are DTI data from the corpus callosum, a structure that connects the two hemispheres of the brain, and on the right, quantiles of this data set, which require the radial fields to be calculated as explained earlier; see \cite{Shin2025} for more details. \cite{Shin2023} and \cite{Shin2025} outline some applications for quantiles on Hadamard manifolds. For example, they explore quantile-based outlier detection, as well as quantile-based measures of dispersion (i.e. spread), skewness (i.e. asymmetry), kurtosis (i.e. tailedness) and spherical asymmetry. These values are calculated for the given DTI data in \cite{Shin2025}; if we can find associations between these distributional characteristics and other variables (such as age or conditions like Alzheimer's disease), these quantile-based measures could potentially be useful in diagnosis. \cite{Shin2025} also investigated the use of quantiles to increase the statistical power of permutation tests (compared to only using the mean/median) in detecting whether two Hadamard space-valued data sets come from the same underlying distribution. In the context of the given DTI data, this could be used, for example, to test whether diffusion in the corpus callosa of two subjects differ significantly. Figure \ref{fig} is not necessarily of much interest in and of itself; it is merely intended to provide a visualization of the quantiles of a DTI data set. However, images like Figure \ref{fig} could still be used to make a crude visual comparison of the dispersion or asymmetry of the DTI data for two different subjects.

\begin{figure}[!t]
	\centering
	\begin{subfigure}[t]{0.49\linewidth}
		\includegraphics[width=\linewidth]{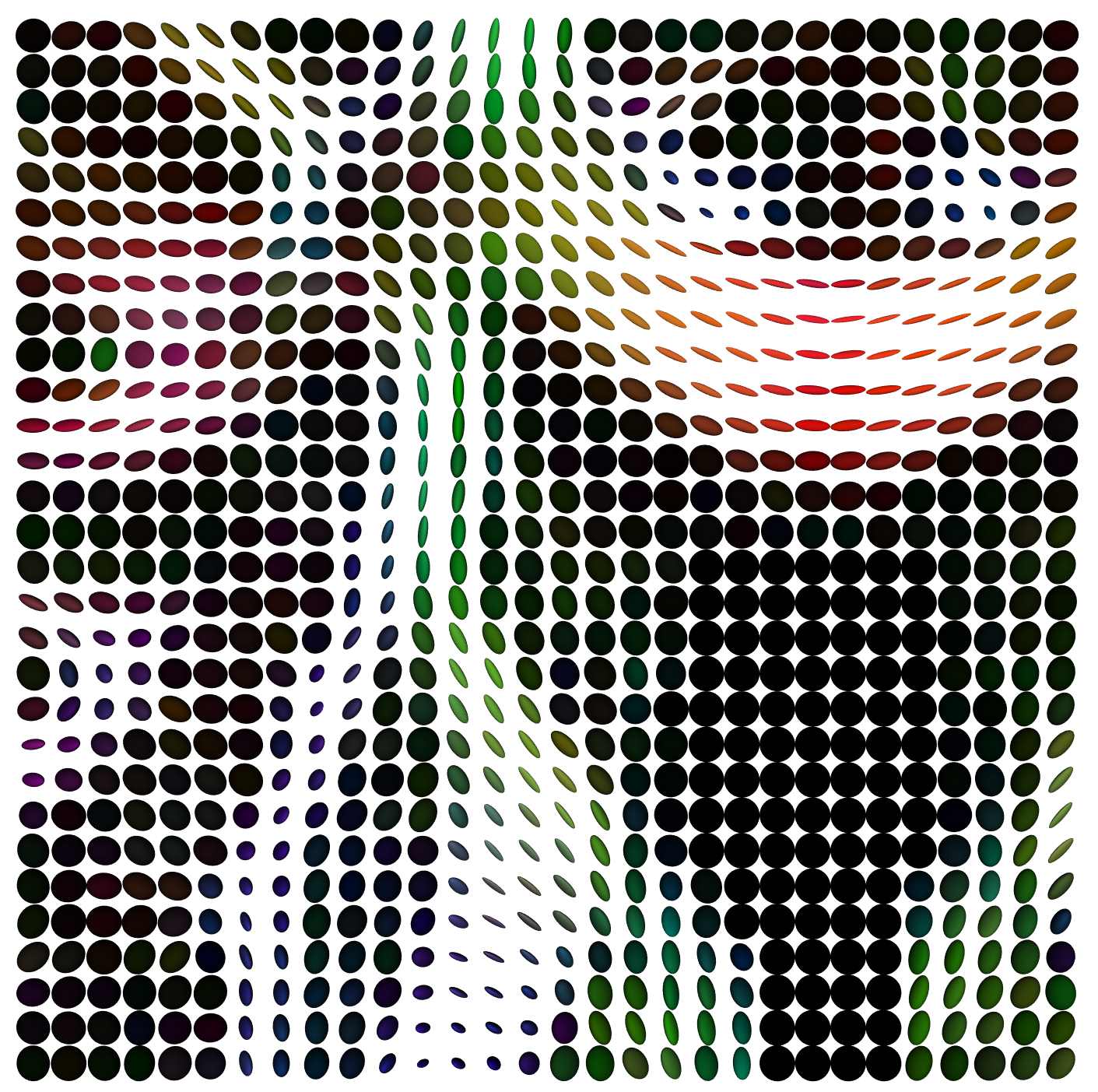}
  \label{dbparta}
	\end{subfigure}
	\begin{subfigure}[t]{0.45\linewidth}
		\includegraphics[width=\linewidth]{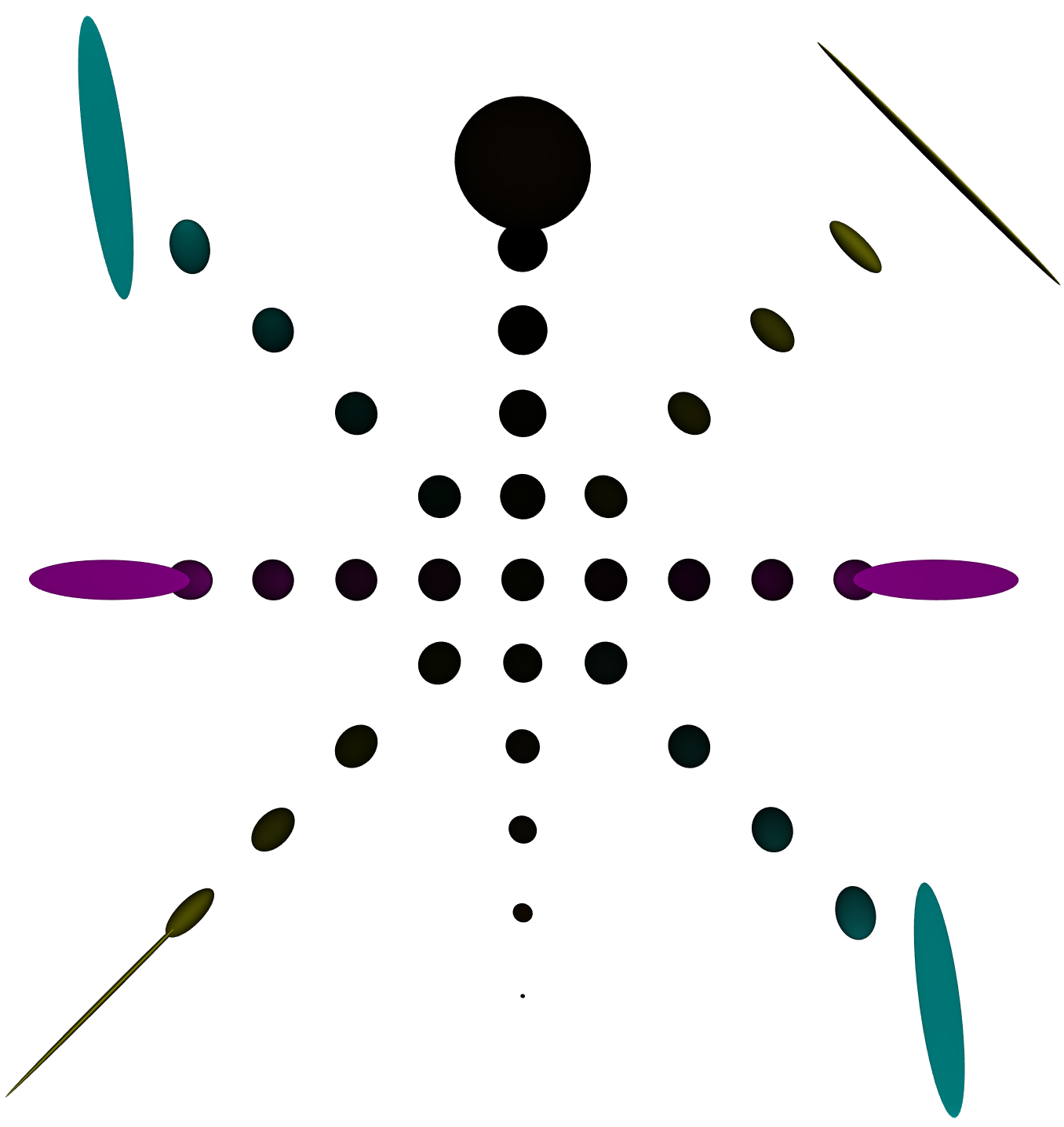}
  \label{dbpartb}
	\end{subfigure}
	\caption{Left: DTI data from the corpus callosum. Right: Quantiles of this data set corresponding to 8 different values of $\xi\in\partial \mathcal{P}_3$; the central ellipsoid represents the median ($\beta=0$), and moving outward in a given direction, the respective ellipsoids represent quantiles indexed by $\beta=0.2,0.4,0.6,0.8,0.98$ and a fixed $\xi$. These images are taken from \cite{Shin2025}, to which we refer the reader for more details.}
	\label{fig}
\end{figure}
 
	\section{Radial fields on $\mathcal{P}_m$}

Any $A\in \mathcal{S}_m$ has a real eigendecomposition $A=V\mathrm{diag}(d_1,\ldots,d_m)V^T$. Then, the matrix exponential of $A$ is SPD and can be written as $\mathrm{Exp}(A)=V\mathrm{diag}(\exp(d_1),\ldots,\exp(d_m))V^T$. Furthermore, if $A$ is also positive semidefinite, it has a unique real $t$th root that is symmetric positive semidefinite:
\begin{equation}\begin{aligned} \label{spdroot}
A^{1/t}=V\mathrm{diag}(d_1^{1/t},\ldots,d_m^{1/t})V^T.
\end{aligned}\end{equation}
Finally, if $A$ is SPD, then $A$ has a unique real SPD matrix logarithm 
\begin{equation}\begin{aligned} \label{spdlog}
\mathrm{Log}(A)=V\mathrm{diag}(\log(d_1),\ldots,\log(d_m))V^T.
\end{aligned}\end{equation}
In this section, all matrices for which we take $t$th roots are symmetric positive semidefinite, and those for which we take logarithms are SPD, so $A^{1/t}$ and $\mathrm{Log}(A)$ will specifically refer to the unique matrices mentioned above. 


The exponential maps and their inverses on $\mathcal{P}_m$ are respectively given by
  \begin{align*}
     \exp_x(v)&=x^{1/2}\mathrm{Exp}(x^{-1/2}vx^{-1/2})x^{1/2},\\
     \log_x(p)&=x^{1/2}\mathrm{Log}(x^{-1/2}px^{-1/2})x^{1/2},
\end{align*}
 where $x,p\in \mathcal{P}_m$ and $v\in T_x\mathcal{P}_m$ (see, for example, Section 3 of \cite{Sra2015}, 3.4 of \cite{Pennec2006}, 5 of \cite{Ferreira2006} or IV.A of \cite{Jaquier2017}), and therefore, the distance between $x$ and $p$ is 
     \begin{align*}
     d(x,p)=\lVert\mathrm{Log}(x^{-1/2}px^{-1/2})\rVert_F,
 \end{align*}
where $\lVert\cdot\rVert_F$ denotes the Frobenius norm. 

In the remainder of this paper, we will denote the Euclidean norm and inner product by $\lVert\cdot\rVert$ and $\langle\cdot,\cdot\rangle$, respectively.

First we state our main result, an expression for the radial field $x\mapsto\xi_x$.

\begin{theorem} \label{main}
    For any $p\in \mathcal{P}_m$ and unit vector $z\in T_p\mathcal{P}_m$, let $\xi$ be the unique point in $\partial\mathcal{P}_m$ satisfying $\xi_p=z$. Take any eigendecomposition $VDV^T$ of $p^{-1/2}\xi_p p^{-1/2}$ satisfying $d_1\geq\cdots\geq d_m$, where $D:=\mathrm{diag}(d_1,\ldots,d_m)$. 

    Take any $x\in \mathcal{P}_m$. Then, denoting the columns of the matrix $W:=x^{-1/2}p^{1/2}V$ by $w_1,\ldots,w_m$ so that $W=[w_1,\ldots,w_m]$, let $u_1,\ldots,u_m$ be an orthonormal basis of $\mathbb{R}^m$ that results from applying the Gram-Schmidt orthonormalization process to $w_1,\ldots,w_m$. Then $\xi_x=x^{1/2}UDU^Tx^{1/2}$, where $U:=[u_1,\ldots,u_m]$.
\end{theorem}

Before proving this theorem, we will prove the following lemma, as well as a related proposition.

\begin{lemma}\label{mainlemma}
    For real $d_1\geq\ldots\geq d_{m}$ and a basis $w_1,\ldots,w_m$ of $\mathbb{R}^m$, define the matrix
\begin{align*}
H(t):=\bigg[\sum_{i=1}^me^{td_i}w_iw_i^T\bigg]^{1/t}. 
\end{align*}

(a) The $j$th largest eigenvalue of $H(t)$ converges to $\exp(d_j)$ $(j=1,\ldots,m)$.

    (b) Let $u_1,\ldots,u_m$ be an orthonormal basis of $\mathbb{R}^m$ that results from applying the Gram-Schmidt orthonormalization process to $w_1,\ldots,w_m$. Then $\lim_{t\rightarrow\infty} (H(t)u_j)=\exp(d_j)u_j$ $(j=1,\ldots,m)$.

    (c) $\lim_{t\rightarrow\infty}H(t)$ exists and has eigenvalues $\exp(d_1),\ldots,\exp(d_m)$ with corresponding eigenvectors $u_1,\ldots,u_m$.
\end{lemma}

The following paragraph applies to the proofs of all parts of this lemma. It will be left unstated that $t$ is restricted to the positive integers, $i$ and $j$ only take values in $\{1,\ldots,m\}$ and $s$ in $1,\ldots,S$, where $S-1$ is defined as the size of the set $\{j\mid d_j\neq d_{j+1}\}\subset\{1,\ldots,m-1\}$; define $n_1<\cdots<n_{S-1}$ to be the elements of this set, $n_0=0$ and $n_S=m$. For any $s$, define $A_s$, $B_s$ and $C_s$ by
\begin{equation*}\begin{aligned}
A_s(t)&:=\sum_{i=1}^{n_s}\exp(t(d_i-d_{n_s}))w_iw_i^T, \\
B_s(t)&:=\sum_{i=n_s+1}^m\exp(t(d_i-d_{n_s}))w_iw_i^T, \\
C_s&:=\sum_{i=1}^{n_s}w_iw_i^T.
\end{aligned}\end{equation*}
For any $j$, denote the $j$th largest eigenvalue of a matrix $A\in \mathcal{S}_m$ by $\alpha_j(A)$.

\begin{proof}[Proof of Lemma \ref{mainlemma}(a)]
Recall Weyl's inequality which states that for $A,B\in\mathcal{S}_m$, $\alpha_{i+j-1}(A+B)\leq\alpha_j(A)+\alpha_i(B)\leq\alpha_{i+j-m}(A+B)$; by letting $i=1$ and $m$, 
\begin{equation}\begin{aligned} \label{weyl}
\alpha_j(A)+\alpha_m(B)\leq\alpha_j(A+B)\leq\alpha_j(A)+\alpha_1(B).
\end{aligned}\end{equation}
Also recall the minimax principle (Section I.10 of \cite{Kato1995}) which states that for $A\in\mathcal{S}_m$, 
\begin{equation}\begin{aligned}\label{minimax}
\alpha_j(A)=\max_{\dim(\mathcal{T})=j}\min_{v\in \mathcal{T},\lVert v\rVert_2=1}v^TAv,
\end{aligned}\end{equation}
where $\mathcal{T}$ is a $j$-dimensional subspace of $\mathbb{R}^m$. Denote by $\mathcal{P}_m'$ the space of $m\times m$ real symmetric positive semidefinite matrices. If $C\in\mathcal{S}_m$ and $A-C\in\mathcal{P}_m'$, $v^TCv=v^TAv-v^T(A-C)v\leq v^TAv$, so by (\ref{minimax}),
\begin{equation}\begin{aligned}\label{ineq}
\alpha_j(C)\leq\alpha_j(A).
\end{aligned}\end{equation}

Since $W$ is invertible, $\{w_1,\ldots,w_m\}$ is indeed a basis for $\mathbb{R}^m$. For any $j$-dimensional subspace $\mathcal{T}$ of $\mathbb{R}^m$, 
\begin{align*}
&\dim(\mathcal{T} \cap \{w_1, \dots, w_{j-1}\}^\perp) \\
=& \dim \mathcal{T} + \dim \{w_1, \dots, w_{j-1}\}^\perp - \dim(\mathcal{T} + \{w_1, \dots, w_{j-1}\}^\perp) \\
\geq& j + m - (j - 1) - m \\
=& 1,
\end{align*}
where $\{w_1, \dots, w_{j-1}\}^\perp$ is the orthogonal complement of the span of $w_1,\ldots, w_{j-1}$. Thus there exists a unit vector $u_\mathcal{T}$ in $\mathcal{T} \cap \{w_1, \dots, w_{j-1}\}^\perp$. 

Set $s$ to be the unique value in $1,\ldots,S-1$ for which $n_{s-1}<j\leq n_s$. Taking the aforementioned $u_{\mathcal{T}}$ gives
\begin{equation}\begin{aligned} \label{ub}
    \alpha_j(A_s(t))&\leq\max_{\dim(\mathcal{T})=j}u_\mathcal{T}^TA_s(t)u_\mathcal{T} \\
    &=\max_{\dim(\mathcal{T})=j}\sum_{i=j}^{n_s}(u_\mathcal{T}^Tw_i)^2 \\
    &\leq \max_{\dim(\mathcal{T})=j}\sum_{i=j}^{n_s}(w_i^Tw_i)(u_\mathcal{T}^Tu_\mathcal{T}) \\
    &=\sum_{i=j}^{n_s}w_i^Tw_i
\end{aligned}\end{equation}
by (\ref{minimax}) and the Cauchy-Schwarz inequality. $A_s(t)-C_s\in\mathcal{P}_m'$ and (\ref{ineq}) holds; since $C_s$ has rank $n_s\geq j$, $\alpha_j(C_s)>0$. Then (\ref{weyl}), (\ref{ub}), and (\ref{ineq}) imply
\begin{equation}\begin{aligned}\label{bounds}
    \alpha_j(A_s(t)+B_s(t))\in\bigg[\alpha_j(C_s)+\alpha_m(B_s(t)),\sum_{i=j}^{n_s}w_i^Tw_i+\alpha_1(B_s(t))\bigg],
\end{aligned}\end{equation}
and because 
\begin{equation}\begin{aligned}\label{limB}
\lim_{t\rightarrow\infty}B_s(t)=0
\end{aligned}\end{equation}
and $\alpha_j(C_s)$ and $\sum_{i=j}^{n_s}w_i^Tw_i$ are finite positive constants independent of $t$, the $t$th root of both bounds in this interval converges to 1 as $t\rightarrow\infty$. Thus, for each $j$, 
\begin{equation}\begin{aligned}\label{evals}
    \lim_{t\rightarrow\infty}\alpha_j(H(t))=\exp(d_j)\lim_{t\rightarrow\infty}(\alpha_j(A_s(t)+B_s(t)))^{1/t}=\exp(d_j).
\end{aligned}\end{equation}
\end{proof}

For $A,B\in\mathcal{S}_m$ and $a,b,\delta\in\mathbb{R}$, let $E$ be a matrix whose columns constitute an orthonormal basis for the eigenspace of $A$ associated with the eigenvalues contained in $(a,b)$, and let $L$ be a matrix whose columns constitute an orthonormal basis for the eigenspace of $A+B$ associated with the eigenvalues contained in $\mathbb{R}\backslash(a-\delta,b+\delta)$. Recall the Davis--Kahan $\sin(\Theta)$ theorem (see Section VII.3 of \cite{Bhatia1996}) which states that
\begin{equation}\begin{aligned}\label{davis}
    \lVert L^TE\rVert_F\leq\frac{\lVert B\rVert_F}{\delta};
\end{aligned}\end{equation}
the norm can be any unitarily invariant norm, the Frobenius norm being one such example. 

In this paragraph, we will give a high-level overview of the proof of \ref{mainlemma}(b). The Davis-Kahan $\sin(\Theta)$ theorem can be used to show that for any $s$, the eigenspace of $H(t)$ corresponding to $\alpha_{1}(H(t)),\ldots,\alpha_{n_s}(H(t))$ converges in some sense to the eigenspace of $A_s(t)$ corresponding to non-zero eigenvalues, which is equivalently the span of $w_1,\ldots,w_{n_s}$ or of $u_1,\ldots,u_{n_s}$; this exploits the fact that $H(t)$ and $A_s(t)+B_s(t)$ have the exact same eigenvectors associated with corresponding eigenvalues thanks to (\ref{spdroot}). Then it can be shown that the eigenspace of $H(t)$ corresponding to $\alpha_{n_{s-1}+1}(H(t)),\ldots,\alpha_{n_s}(H(t))$ converges to the span of $u_{n_{s-1}+1},...,u_{n_s}$.

\begin{proof}[Proof of Lemma \ref{mainlemma}(b)]
Letting $a=\alpha_{n_s}(C_s)/2>0$, $b=\sum_{i=1}^{n_s}w_i^Tw_i+1<\infty$ and $\delta=\alpha_{n_s}(C_s)/4$, then $\alpha_i(A_s(t))\in(a,b)$ precisely when $i=1,\ldots,n_s$, thanks to (\ref{ineq}), (\ref{ub}), and the fact that $m-n_s$ of the eigenvalues of $A_s(t)$ are 0 because $A_s(t)$ has rank $n_s$. The eigenspace associated with these eigenvalues is precisely the span of $\{w_1,\ldots,w_{n_s}\}$, and so we can let $E$ in (\ref{davis}) be
\begin{align*}
E_s:=[u_1,\ldots,u_{n_s}],
\end{align*}
$(s=1,\ldots,S)$.

Denote an ordered orthonormal basis of $\mathbb{R}^m$ of eigenvectors of $A_s(t)+B_s(t)$ by $y_{1}(t),\ldots,y_{m}(t)$. For any $r$ and $s$ satisfying $r\in\{0,\ldots,S-1\}$ and $r<s$, define $K_{r,s}$ by $K_{r,s}(t):=[y_{{n_{r}}+1}(t),\ldots,y_{n_{s}}(t)]$. For any $s$, define $L_s$ by $L_s(t):=[y_{n_s+1}(t),\ldots,y_{m}(t)]$. $H(t)$ and $A_s(t)+B_s(t)$ have the exact same eigenvectors so each $y_i(t)$ does not depend on $s$. For $i$ in $\{n_s+1,\ldots,m\}$, letting $s'$ be the unique integer for which $n_{s'-1}<i\leq n_{s'}$, $\alpha_i(A_s(t)+B_s(t))=\exp(t(d_{i}-d_{n_s}))\alpha_i(A_{s'}(t)+B_{s'}(t))\rightarrow 0$ as $t\rightarrow \infty$ by (\ref{bounds}) and (\ref{limB}) since $d_{i}<d_{n_s}$. Therefore, $\alpha_{n_s+1}(A_s(t)+B_s(t)),\ldots,\alpha_m(A_s(t)+B_s(t))\in\mathbb{R}\backslash(a-\delta,b+\delta)$ for sufficiently large $t$ and we can choose $L$ in (\ref{davis}) such that $y_{n_s+1}(t),\ldots,y_{m}(t)$ are among its columns.

Recall that for any matrix $X$ with linearly independent columns, $X^TX$ is invertible and the projection matrix, defined as $X(X^TX)^{-1}X^T$, projects a vector into the column space of $X$; in particular, if the columns of $X$ are also orthogonal, the projection matrix is $XX^T$. For any $i$ and $s$, define $v_{s,i}(t):=K_{0,s}(t)K_{0,s}(t)^Tu_i$, the projection of $u_i$ onto the span of $y_1(t),\ldots,y_{n_s}(t)$. If $i\in\{1,\ldots,n_s\}$, $v_{s,i}(t)$ satisfies
\begin{equation}\begin{aligned} \label{ui}
    \lVert u_i-v_{s,i}(t)\rVert_2&=\lVert(I-K_{0,s}(t)K_{0,s}(t)^T)u_i\rVert_2  \\
    &=\lVert L_s(t)L_s(t)^Tu_i\rVert_2 \\
    &=(u_i^TL_s(t)L_s(t)^TL_s(t)L_s(t)^Tu_i)^{1/2} \\
    &=\lVert L_s(t)^Tu_i\rVert_2 \\
    &\rightarrow 0 
\end{aligned}\end{equation}
as $t\rightarrow\infty$ by (\ref{limB}) and (\ref{davis}) since $u_i$ is a column of $E_s$. For any $s$, define $V_s$ by $V_s(t):=[v_{s,1}(t),\ldots,v_{s,n_s}(t)]$. The columns of $V_s(t)$ form a basis of the column space of $K_{0,s}(t)$ when $t$ is sufficiently large because (\ref{ui}) ensures that they are eventually linearly independent. Thus for sufficiently large $t$ the projection matrices of $V_{s}(t)$ and $K_{0,s}(t)$ are equal, and therefore (\ref{ui}) implies that for all $s$,
\begin{align*}
    \lim_{t\rightarrow\infty}K_{0,s}(t)K_{0,s}(t)^T&=\lim_{t\rightarrow\infty}V_{s}(t)(V_{s}(t)^TV_{s}(t))^{-1}V_{s}(t)^T \\
    &=E_{s}(E_{s}^TE_{s})^{-1}E_{s}^T \\
    &=E_{s}E_{s}^T.
\end{align*}
This means that for any $s=2,\ldots,S$, if $i\in\{n_{s-1},\ldots,m\}$, $v_{s-1,i}(t)=K_{0,s-1}(t)K_{0,s-1}(t)^Tu_i\rightarrow E_{s-1}E_{s-1}^Tu_i=0$ as $t\rightarrow\infty$. Also, for the same values of $s$, 
\begin{align*}
v_{s,i}(t)&=K_{0,s}(t)K_{0,s}(t)^Tu_i \\
&=\begin{bmatrix}K_{0,s-1}(t) & K_{s-1,s}(t)\end{bmatrix}\begin{bmatrix}K_{0,s-1}(t)^T  \\ K_{s-1,s}(t)^T\end{bmatrix}u_i \\
&=v_{s-1,i}(t)+K_{s-1,s}(t)K_{s-1,s}(t)^Tu_i.
\end{align*}
So for any $s=2,\ldots,S$ (the $s=1$ case will not be needed), if $l\in\{1,\ldots,n_{s-1}\}$ and $i\in\{n_{s-1}+1,\ldots,n_s\}$,
\begin{equation}\begin{aligned}\label{vsi1}
    \lvert\langle v_{s,i}(t),y_l(t)\rangle_2\rvert&\leq\lvert\langle v_{s-1,i}(t),y_l(t)\rangle_2\rvert+\lvert u_i^TK_{s-1,s}(t)K_{s-1,s}(t)^Ty_l(t)\rvert \\
    &\leq\lVert v_{s-1,i}(t)\rVert_2+0 \\
    &\rightarrow 0
\end{aligned}\end{equation}
as $t\rightarrow\infty$. In addition, for any $s$,
\begin{equation}\begin{aligned}\label{vsi2}
    \langle v_{s,i}(t),y_l(t)\rangle_2=0
\end{aligned}\end{equation}
when $l\in\{n_s+1,\ldots,m\}$ because $v_{s,i}(t)$ is in the span of $y_1(t),\ldots,y_{n_s}(t)$. So, keeping (\ref{evals}), (\ref{ui}), (\ref{vsi1}) and (\ref{vsi2}) in mind, for any $s$ and $i=n_{s-1}+1,\ldots,n_s$,
\begin{align*}
    \lim_{t\rightarrow\infty} \big(H(t)u_i\big)&=\lim_{t\rightarrow\infty} H(t)\bigg(\sum_{l=1}^{m}\langle u_i,y_{l}(t)\rangle_2 y_{l}(t)\bigg) \\
    &=\lim_{t\rightarrow\infty}\bigg(\sum_{l=1}^{m}\alpha_l(H(t))\langle u_i,y_{l}(t)\rangle_2y_{l}(t)\bigg) \\
&=\lim_{t\rightarrow\infty}\bigg(\sum_{l=n_{s-1}+1}^{n_s}\exp(d_{n_s})\langle v_{s,i}(t),y_{l}(t)\rangle_2y_{l}(t)\bigg) \\
    &\qquad+\lim_{t\rightarrow\infty}\bigg(\sum_{l=n_{s-1}+1}^{n_s}[\alpha_l(H(t))-\exp(d_{n_s})]\langle v_{s,i}(t),y_{l}(t)\rangle_2y_{l}(t)\bigg) \\
    &\qquad+\lim_{t\rightarrow\infty}\bigg(\sum_{l=1}^{n_{s-1}}\alpha_l(H(t))\langle v_{s,i}(t),y_{l}(t)\rangle_2y_{l}(t)\bigg) \\&\qquad+\lim_{t\rightarrow\infty}\bigg(\sum_{l=n_s+1}^{m}\alpha_l(H(t))\langle v_{s,i}(t),y_{l}(t)\rangle_2y_{l}(t)\bigg) \\
&\qquad+\lim_{t\rightarrow\infty}\bigg(\sum_{l=1}^{m}\alpha_l(H(t))\langle u_i-v_{s,i}(t),y_{l}(t)\rangle_2y_{l}(t)\bigg) \\
&=\exp(d_{n_s})\lim_{t\rightarrow\infty}v_{s,i}(t) \\
&=\exp(d_{n_s})u_i.
\end{align*}
\end{proof}

Notice that at no point in the proofs of Lemma \ref{mainlemma}(a) and (b) is $\lim_{t\rightarrow\infty}H(t)$ assumed to exist, and in fact we will use (b) to prove that it exists and that its eigenvalues and eigenvectors are those suggested by (a) and (b).

\begin{proof}[Proof of Lemma \ref{mainlemma}(c)]
    Defining $U:=[u_1,\ldots,u_m]$, (b) shows that
    \begin{align*}
        \lim_{t\rightarrow\infty}H(t)=\lim_{t\rightarrow\infty}\big(H(t)U\big)U^{-1}=U\mathrm{diag}(\exp(d_1),\ldots,\exp(d_m))U^{-1}.
    \end{align*}
\end{proof}

Lemma \ref{mainlemma} admits the following generalization, which does not require the given linearly independent vectors to span the entire ambient Euclidean space; this result is not needed in our main proof, but is included for the sake of general interest.

\begin{proposition}\label{prop}
    For real $d_1\geq\ldots\geq d_{m}$ and linearly independent $w_1,\ldots,w_{m}\in\mathbb{R}^{m'}$, where $m'\geq m$, let $u_1,\ldots,u_{m}$ be an orthonormal set that results from applying the Gram-Schmidt orthonormalization process to $w_1,\ldots,w_{m}$, and $u_{m'+1},\ldots,u_{m'}$ be an orthonormal basis for the orthogonal complement of the span of $w_1,\ldots,w_{m}$. Then the matrix
\begin{align*}
\lim_{t\rightarrow\infty}\bigg[\sum_{i=1}^{m}e^{td_i}w_iw_i^T\bigg]^{1/t} 
\end{align*}
exists, has eigenvalues $\exp(d_1),\ldots,\exp(d_{m}),0,\ldots,0$, where $m'-m$ of the eigenvalues are 0, and has corresponding orthonormal eigenvectors $u_1,\ldots,u_{m},u_{m+1},\ldots,u_{m'}$.
\end{proposition}
\begin{proof}
    Define the matrices $U_1:=[u_1,\ldots,u_m]$ and $U_2:=[u_1,\ldots,u_m,u_{m+1},\ldots,u_{m'}]$. For any positive integers $k,l$, define $O_{kl}$ to be the $k\times l$ zero matrix. 
    
    Clearly $\sum_{i=1}^{m}e^{td_i}w_iw_i^T$ is symmetric positive semidefinite, so $U_2^T[\sum_{i=1}^{m}\exp(td_i)w_iw_i^T]U_2$ is as well, and therefore it has a unique $t$th root that is symmetric positive semidefinite. Because 
    \begin{align*}
    \bigg(U_2^T\bigg[\sum_{i=1}^{m}e^{td_i}w_iw_i^T\bigg]^{1/t}U_2\bigg)^t=U_2^T\bigg[\sum_{i=1}^{m}e^{td_i}w_iw_i^T\bigg]U_2,
    \end{align*}
    this unique $t$th root is given by
    \begin{align}\label{tthroot}
        \bigg[\sum_{i=1}^{m}e^{td_i}(U_2^Tw_i)(U_2^Tw_i)^T\bigg]^{1/t}=U_2^T\bigg[\sum_{i=1}^{m}e^{td_i}w_iw_i^T\bigg]^{1/t}U_2.
    \end{align}
    
    Recall that any vector $v\in\mathbb{R}^{m'}$ can be can be expressed with respect to the basis $u_1,\ldots,u_{m'}$ as $U_2^Tv$, and notice that if $v$ is in the span of $w_1,\ldots,w_m$, the last $m'-m$ entries of $U_2^Tv$ are zero; hence 
    \begin{align}\label{orth}
    U_2^Tv=\begin{bmatrix}U_1^Tv \\ 0 \\ \vdots \\0\end{bmatrix},
    \end{align}
    ending with $m'-m$ zeroes. So for any $v_1,v_2$ in the span of $w_1,\ldots,w_m$,
    \begin{align*}
    \langle v_1,v_2\rangle=v_1^TU_2U_2^Tv_2=\begin{bmatrix}v_1^TU_1 & 0 & \ldots & 0 \end{bmatrix}\begin{bmatrix}U_1^Tv_2 \\ 0 \\ \vdots \\0\end{bmatrix}=v_1^TU_1U_1^Tv_2=\langle U_1^Tv_1,U_1^Tv_2\rangle,
    \end{align*}
    so the Gram-Schmidt orthonormalization of $U_1^Tw_1,\ldots,U_1^Tw_m$ is $U_1^Tu_1,\ldots,U_1^Tu_m$, that is, the standard $m$-dimensional basis $e_1,\ldots,e_m$. Therefore, by Lemma \ref{mainlemma}(c), the matrix $\lim_{t\rightarrow\infty}[\sum_{i=1}^m\exp(td_i)(U_1^Tw_i)(U_1^Tw_i)^T]^{1/t}$ exists, has eigenvalues $\exp(d_1),\ldots,\exp(d_m)$, and corresponding eigenvectors $e_1,\ldots,e_m$. Thus, using (\ref{tthroot}) and letting $v=w_i$ in (\ref{orth}) gives
    \begin{align*}
        &\lim_{t\rightarrow\infty}\bigg[\sum_{i=1}^me^{td_i}w_iw_i^T\bigg]^{1/t} \\
        =&U_2\bigg(\lim_{t\rightarrow\infty}U_2^T\bigg[\sum_{i=1}^me^{td_i}w_iw_i^T\bigg]^{1/t}U_2\bigg)U_2^T \\
        =&U_2\bigg(\lim_{t\rightarrow\infty}\bigg[\sum_{i=1}^{m}e^{td_i}(U_2^Tw_i)(U_2^Tw_i)^T\bigg]^{1/t}\bigg)U_2^T \\
        =&U_2\begin{bmatrix}\lim_{t\rightarrow\infty}[\sum_{i=1}^m\exp(td_i)(U_1^Tw_i)(U_1^Tw_i)^T]^{1/t} & O_{m,m'-m} \\
        O_{m'-m,m} & O_{m'-m,m'-m}\end{bmatrix}U_2^T \\
        =&U_2\begin{bmatrix}\mathrm{diag}(\exp(d_1),\ldots,\exp(d_m)) & O_{m,m'-m} \\
        O_{m'-m,m} & O_{m'-m,m'-m}\end{bmatrix}U_2^T \\
        =&U_2\mathrm{diag}(\exp(d_1),\ldots,\exp(d_m),0,\ldots,0)U_2^T.
    \end{align*}
\end{proof}

We now prove our main result.

\begin{proof}[Proof of Theorem \ref{main}]
    The limit $\lim_{t\rightarrow\infty}d(x,\exp_p(t\xi_p))/t=1$ since 
    \begin{align*}
    (d(\exp_p(t\xi_p),p)-d(x,p))/t\leq d(x,\exp_p(t\xi_p))/t\leq (d(\exp_p(t\xi_p),p)+d(x,p))/t
    \end{align*}
    for positive $t$ by the triangle equality, so
\begin{equation}\begin{aligned} \label{back}
    \xi_x&=\lim_{t\rightarrow\infty}\frac{\log_x(\exp_p(t\xi_p))}{d(x,\exp_p(t\xi_p))} \\
    &=\lim_{t\rightarrow\infty}\frac{x^{1/2}\mathrm{Log}(x^{-1/2}p^{1/2}\mathrm{Exp}(tp^{-1/2}\xi_p p^{-1/2})p^{1/2}x^{-1/2})x^{1/2}}{t} \\
    &=x^{1/2}(\lim_{t\rightarrow\infty}\mathrm{Log}\big([x^{-1/2}p^{1/2}\mathrm{Exp}(tp^{-1/2}\xi_p p^{-1/2})p^{1/2}x^{-1/2}]^{1/t})\big)x^{1/2} \\
    &=x^{1/2}\mathrm{Log}\bigg(\lim_{t\rightarrow\infty}[x^{-1/2}p^{1/2}\mathrm{Exp}(tp^{-1/2}\xi_p p^{-1/2})p^{1/2}x^{-1/2}]^{1/t}\bigg)x^{1/2} \\
    &=x^{1/2}\mathrm{Log}\bigg(\lim_{t\rightarrow\infty}[x^{-1/2}p^{1/2}V\mathrm{Exp}(tD)V^Tp^{1/2}x^{-1/2}]^{1/t}\bigg)x^{1/2} \\
    &=x^{1/2}\mathrm{Log}\bigg(\lim_{t\rightarrow\infty}\bigg[\sum_{i=1}^me^{td_i}w_iw_i^T\bigg]^{1/t}\bigg)x^{1/2};
\end{aligned}\end{equation}
the first equality follows from (\ref{xi}), and the limit in the fourth exists because it must equal $\mathrm{Exp}(x^{-1/2}\xi_xx^{-1/2})$ by the continuity of $\mathrm{Exp}$; alternatively, it exists because it must equal the limit in the last line, which exists by Lemma \ref{mainlemma}(c). The result follows immediately by the same result.
\end{proof}

One may wonder whether the radial fields are smooth on Hadamard manifolds. In fact, though they are known to be $C^1$ (see Proposition 3.1 of \cite{Heintze1977}) they are not guaranteed to even be $C^2$. \cite{Green1974} and \cite{Shcherbakov1983} provide some conditions on the curvature of the manifold under which twice continuous differentiability can be guaranteed, but since they require the supremum of the sectional curvatures to be less than 0, these results do not apply to $\mathcal{P}_m$. However, we can show that the radial fields are, in fact, smooth in $\mathcal{P}_m$, just as \cite{Shin2023} did in hyperbolic spaces.

\begin{corollary}
    The radial fields on $\mathcal{P}_m$ are smooth.
\end{corollary}
\begin{proof}
    Because $z\mapsto z/\lVert z\rVert_2$ on $\mathbb{R}^m\backslash\{0\}\rightarrow \mathbb{R}^m\backslash\{0\}$ is smooth, $W\mapsto U$ defined on $\text{GL}_m(\mathbb{R})\rightarrow \text{GL}_m(\mathbb{R})$, which is diffeomorphic to an open subset of $\mathbb{R}^{m^2}\backslash\{0\}$, is also smooth. The map $z\mapsto z^{1/2}$ on $\mathcal{P}_m\rightarrow \mathcal{P}_m$, diffeomorphic to an open subset of $\mathbb{R}^{m(m+1)/2}\backslash\{0\}$, is also smooth, and therefore so is $x\mapsto W$ on $\mathcal{P}_m\rightarrow\text{GL}_m(\mathbb{R})$. Then, the smoothness of the map $x\mapsto \xi_x$ on $\mathcal{P}_m\rightarrow\mathcal{S}_m\cong \mathbb{R}^{m(m+1)/2}$ follows from Theorem \ref{main}.
\end{proof}

This smoothness is important because, for example, it means that the joint asymptotic normality of quantiles of Theorem 4.2, Corollary 4.1 and Proposition 4.2 of \cite{Shin2023} can be applied to quantiles on $\mathcal{P}_m$, and that the gradient of the quantile loss functions in that space can also be calculated using Jacobi fields as in hyperbolic spaces.

	\section{Concluding remarks}
	
	As detailed in the introduction, radial fields have the potential to generalize many statistical techniques to Hadamard manifolds by defining a canonical sense of direction. The results of this paper, namely an expression for the radial fields on $\mathcal{P}_m$, among the most commonly encountered Hadamard manifolds, and the smoothness of these fields, should be of great use to researchers looking to apply these techniques to $\mathcal{P}_m$.
	
	\section*{Acknowledgments} 
		I would like to thank my advisor Hee-Seok Oh for his helpful comments on this manuscript. 

    \section*{Funding}
        This research was supported by the G-LAMP Program of the National Research Foundation of Korea (NRF) grant funded by the Ministry of Education (No. RS-2025-25441317).

 \bibliographystyle{apalike}
\bibliography{references}

\end{document}